\documentclass[11pt,reqno,letterpaper]{amsart}

\usepackage{booktabs}
\usepackage{color}
\usepackage[colorlinks=true, allcolors=blue,backref=page,hypertexnames=false]{hyperref}
\usepackage{amsmath, amssymb, amsthm}
\usepackage{mathrsfs}
\usepackage{mathtools}
\usepackage[noabbrev,capitalize,nameinlink]{cleveref}
\usepackage[noadjust]{cite}
\usepackage{graphics}
\usepackage{pifont}
\usepackage{tikz}
\usepackage{bbm}
\usepackage{thm-restate}
\usepackage[T1]{fontenc}

\usetikzlibrary{arrows.meta}

\usepackage{environ}
\usepackage{framed}
\usepackage{url}
\usepackage[linesnumbered,ruled,vlined]{algorithm2e}
\usepackage[noend]{algpseudocode}
\usepackage[labelfont=bf]{caption}
\usepackage[framemethod=tikz]{mdframed}
\usepackage{appendix}
\usepackage{graphicx}
\usepackage[textsize=tiny]{todonotes}
\usepackage{tcolorbox}
\usepackage{enumerate}
\allowdisplaybreaks[1]
\usepackage{stmaryrd}

\usepackage[margin=1in]{geometry}

\usepackage[shortlabels]{enumitem}
\crefformat{enumi}{#2#1#3}
\crefrangeformat{enumi}{#3#1#4 to~#5#2#6}
\crefmultiformat{enumi}{#2#1#3}
{ and~#2#1#3}{, #2#1#3}{ and~#2#1#3}

\crefname{algocf}{Algorithm}{Algorithms}

\crefname{equation}{}{} 
\AtBeginEnvironment{appendices}{\crefalias{section}{appendix}} 

\usepackage[color,final]{showkeys} 

\colorlet{refkey}{orange!20}
\colorlet{labelkey}{blue!30}

\newcommand{\N}{\mathbb{N}}

\numberwithin{equation}{section}
\newtheorem{theorem}{Theorem}[section]
\newtheorem{proposition}[theorem]{Proposition}
\newtheorem{lemma}[theorem]{Lemma}

\newtheorem{observation}[theorem]{Observation}
\crefname{claim}{Claim}{Claims}

\newtheorem{conjecture}[theorem]{Conjecture}
\newtheorem*{question*}{Question}

\theoremstyle{definition}
\newtheorem{definition}[theorem]{Definition}

\newtheorem*{definition*}{Definition}
\newtheorem{example}[theorem]{Example}

\newtheorem{remark}[theorem]{Remark}

\newcommand{\paren}[1]{\left(#1\right)}

\newcommand{\mb}{\mathbb}
\newcommand{\mbf}{\mathbf}

\newcommand{\mc}{\mathcal}

\let\originalleft\left
\let\originalright\right
\renewcommand{\left}{\mathopen{}\mathclose\bgroup\originalleft}
\renewcommand{\right}{\aftergroup\egroup\originalright}

\allowdisplaybreaks

\newcommand{\ignore}[1]{}

\DeclareMathOperator{\RMMS}{RMMS}
\DeclareMathOperator{\MMS}{MMS}

\DeclareMathOperator{\Bin}{Bin}
\newcommand{\V}{\mathcal V}
\newcommand{\Vzeroone}{\mathcal V_{01}}
\newcommand{\ind}{\mathbf 1}

\newcommand{\PP}{\mathbb P}

\title{Thinned Quantile Shares are Universally Feasible}

\author{Vishesh Jain}
\address{Department of Mathematics, Statistics, and Computer Science, University of Illinois Chicago, Chicago, IL 60607, USA}
\email{visheshj@uic.edu}

\author{Clayton Mizgerd}
\address{Department of Mathematics, Statistics, and Computer Science, University of Illinois Chicago, Chicago, IL 60607, USA}
\email{cmizge2@uic.edu}

\author{Shyam Ravichandran}
\address{Department of Mathematics, Statistics, and Computer Science, University of Illinois Chicago, Chicago, IL 60607, USA}
\email{sravi42@uic.edu}

\begin{document}

\begin{abstract}
Quantile shares, introduced by Babichenko, Feldman, Holzman, and Narayan~[STOC 2024], offer an ordinal, self-maximizing, and interpretable benchmark for fair division of indivisible goods, but their universal feasibility is known only conditional on the rainbow Erd\H{o}s matching conjecture (EMC). Specifically, Babichenko et al.~showed that assuming the rainbow EMC in the near-perfect matching regime, the $(1/2e)$-quantile share is universally feasible. In contrast, a simple argument shows that the $q$-quantile share can be infeasible for any $q > 1/e$. 

We introduce a one-parameter refinement of quantile shares, the \emph{$c$-thinned quantile share}, obtained by thinning the inclusion probability in the random benchmark bundle by a factor of $c$ for a fixed constant $c\in(0,1]$. Our main result is that there exists a universal constant $c >0$ for which the $c$-thinned $e^{-c}$-quantile share is unconditionally universally feasible; this is best possible in the sense that for any $c \in (0,1]$, the $c$-thinned $q$-quantile share can be infeasible for any $q > e^{-c}$. Prior to this work, the only nontrivial share known to be universally feasible was Feige's residual maximin share. 

The thinning viewpoint also lets us remove the factor-two loss in the conditional result for the original quantile share: assuming the rainbow EMC, the $(1/e)$-quantile share is universally feasible. 
\end{abstract}

\maketitle

\section{Introduction}

Fair division concerns the allocation of goods among agents in a fair manner. For indivisible goods, two broad approaches have been especially influential. In the \emph{envy-based} approach, fairness is defined by comparing an agent's bundle to the bundles received by others. An allocation is \emph{envy-free} if no agent prefers another agent's bundle to her own; because exact envy-freeness often fails for indivisible goods, much of the literature studies relaxations such as EF1 and EFX; we refer the reader to the survey \cite{AzizLMW22}. A second major approach is \emph{share-based} fairness \cite{BF}, in which each agent is assigned a benchmark value depending only on her own valuation and the number of agents, and an allocation is fair if every agent receives at least her benchmark. This paper belongs to the latter tradition.

We begin by fixing the basic model and terminology. Let $[m]=\{1,\dots,m\}$ be a set of indivisible goods and let $n\ge 2$ be the number of agents. A \emph{valuation} is a function
\[
v:2^{[m]}\to \mathbb{R}_{\ge 0}.
\]
A valuation is said to be \emph{monotone} if
 $v(\emptyset)=0$ and $v(A)\le v(B)$ whenever $A\subseteq B$. Let $\mathcal{V}$ denote the class of all such monotone valuations, and let $\mathcal{V}_{01}\subseteq \mathcal{V}$ denote the subclass of monotone $0/1$-valuations (that is, monotone valuations taking values in $\{0,1\}$). An \emph{allocation} is an $n$-tuple $(S_1,\dots,S_n)$ of pairwise disjoint bundles whose union is $[m]$.

\begin{definition}
A \emph{share} is a rule $\tau$ assigning a nonnegative real number $\tau(v,n)$ to every pair $(v,n)$ consisting of a valuation and a number of agents. An allocation $(S_1,\dots,S_n)$ is \emph{fair} for a valuation profile $(v_1,\dots,v_n)$ with respect to $\tau$ if
\[
v_i(S_i)\ge \tau(v_i,n)\qquad\text{for every }i\in[n].
\]
A share $\tau$ is \emph{feasible} for a class $\mathcal{U}$ of valuations if every profile of valuations in $\mathcal{U}$ admits a fair allocation, and it is \emph{universally feasible} if it is feasible for all monotone valuations, i.e., for $\mathcal{U}=\mathcal{V}$.
\end{definition}

Two of the most classical shares are the \emph{proportional share} \cite{Steinhaus48} and the \emph{maximin share} \cite{Budish11}:
\[
\operatorname{PS}(v,n):=\frac{v([m])}{n},
\qquad
\operatorname{MMS}(v,n):=\max_{(P_1,\dots,P_n)} \min_{j\in[n]} v(P_j),
\]
where the maximum is taken over all partitions $(P_1,\dots,P_n)$ of $[m]$ into $n$ bundles. Proportionality is the benchmark suggested by exact divisibility, while MMS asks what an agent can guarantee by partitioning the goods into $n$ bundles however she wants and then receiving the least desirable bundle.

These shares are natural, but they are not universally feasible. The proportional share is not feasible in general indivisible-goods settings, and MMS is not feasible even for additive valuations~\cite{KPW18}. This has led to a large literature on constant-factor approximations to MMS and related shares. For instance, one now knows that a $(3/4+3/3836)$-fraction of MMS is feasible for additive valuations \cite{AG24}; for submodular and other richer valuation classes, constant-factor approximations are also known~\cite{BK20,GHSY22}. However, for general monotone valuations this route cannot succeed: Babichenko, Feldman, Holzman, and Narayan~\cite{BFHN} show that no constant fraction of the proportional or MMS share is universally feasible. Thus, if one seeks a universally feasible share for arbitrary monotone valuations, scaling down the proportional or MMS benchmarks is not enough; one needs a different share method.

\medskip

More recently, Babaioff and Feige~\cite{BF} developed a general framework for share-based fairness and identified several structural desiderata for shares. Among them are \emph{realizability} (the share should never exceed the value of the whole set of goods), \emph{invariance under renaming goods}, and \emph{self-maximization}, meaning that truthful reporting maximizes the worst-case true value of a bundle deemed acceptable under the reported valuation. Within this framework, two especially promising candidates have emerged for universal feasibility: the \emph{quantile-share} family introduced by Babichenko et al.~\cite{BFHN}, and the \emph{residual maximin share (RMMS)} introduced by Feige~\cite{FeigeRMMS}. We remark that apart from not being universally feasible, neither the proportional share nor any constant $\rho < 1$ fraction of MMS is self-maximizing \cite{BF}.

\medskip

The starting point of the present paper is the quantile-share paradigm of~\cite{BFHN}. Let $X\subseteq [m]$ denote the random bundle obtained by including each good independently with probability $1/n$. The \emph{$q$-quantile share} of an agent with valuation $v$ is the $q$-quantile of the random value $v(X)$; equivalently,
\[
\tau_q(v,n):=\min\bigl\{t\in v(2^{[m]}): \mathbb{P}[v(X)\le t]\ge q\bigr\}.
\]
Thus a bundle $S$ is $q$-fair if it is at least as good as the random benchmark bundle with probability at least $q$. Quantile shares have several appealing features: they are \emph{ordinal}, in the sense that they depend only on the induced ranking of bundles rather than on the cardinal scale of utilities, and they are self-maximizing~\cite{BFHN}. Moreover, compared to notions such as MMS, which are NP-hard to approximate to any factor, even on instances where they are feasible, quantile shares can be approximated using Monte Carlo methods. We remark that quantile shares may be viewed as an indivisibility-aware analogue of proportionality: for additive valuations, replacing the quantile by expectation recovers the proportional share.

While quantile-type criteria had appeared earlier in certain two-agent allocation settings (see the discussion in \cite{BFHN}), Babichenko et al.~\cite{BFHN} introduced the $q$-quantile share as a share notion in the modern fair-division framework and initiated its systematic study for monotone valuations, with the explicit goal of finding a nontrivial universally feasible share. Their main result establishes a striking connection to extremal set theory: if a rainbow version of the celebrated Erd\H{o}s Matching Conjecture holds, then the quantile share is universally feasible at level $q=1/(2e)$. We remark that the Erd\H{o}s Matching Conjecture is one of the central open problems in extremal set theory, and is also known only in partial regimes; we return to this topic in \cref{sec:rainbow}. Babichenko et al.\ also prove that the $q>1/e$-quantile is not universally feasible. They further obtain unconditional constant-$q$ feasibility results for several restricted classes, including additive, unit-demand, and matroid-rank valuations. Thus their work leaves open both the main unconditional question---whether some constant value of $q$ yields universal feasibility using only known results---and, more specifically, whether the factor-two gap between the conditional value $1/(2e)$ and the upper barrier $1/e$ can be closed \cite[Section~5.1]{BFHN}. 

\subsection{Our results}

We show that the quantile-share paradigm admits an \emph{unconditional} universally feasible share, after a conceptually simple modification of the benchmark. Rather than comparing an agent's allocation to the usual random bundle in which each good is included independently with probability $1/n$, we compare it to a \emph{thinned} random bundle; see \cref{sec:thinning-motivation} for motivation.

\begin{definition}
Fix $c\in(0,1]$. The \emph{$c$-thinned random bundle} $X^{(c)}\subseteq [m]$ is obtained by including each good independently with probability $c/n$. 
\end{definition}

\begin{definition}
Fix $c\in(0,1]$, $q\in(0,1)$, and $v\in\mathcal{V}$. The \emph{$c$-thinned $q$-quantile share} is
\[
\tau_q^{(c)}(v,n):=\min\bigl\{t\in v(2^{[m]}): \mathbb{P}[v(X^{(c)})\le t]\ge q\bigr\}.
\]
When $c=1$, this is exactly the original quantile share $\tau_q(v,n)$ of~\cite{BFHN}.
\end{definition}
\begin{remark}
For additive valuations, the random bundle $X^{(c)}$ has expected value
\[
\mathbb{E}[v(X^{(c)})]=c\,\operatorname{PS}(v,n),
\]
so the thinned benchmark may be viewed as a quantile analogue of $c$ times proportional share.
\end{remark}

Our first main result is an unconditional universal-feasibility theorem.

\begin{theorem}\label{thm:main}
There exists a universal constant $c >0$ such that the $c$-thinned $e^{-c}$-quantile share is universally feasible. 
\end{theorem}

\begin{remark}
While we have not made an attempt to optimize constants, we note that one can take $c = 1/250$ for all $n\geq 2$, and $c = 1/10$ for all sufficiently large $n$. We state a sharper quantitative version of \cref{thm:main} in \cref{thm:quantitative}.    
\end{remark}

Interestingly, the thinning viewpoint also allows us to eliminate the factor-two gap in the conditional universal feasibility result of \cite{BFHN} for the (unthinned) quantile share.

\begin{theorem}\label{thm:oneovere}
Assume the rainbow Erd\H{o}s Matching Conjecture. Then the $q$-quantile share is universally feasible for
\[
q=\frac{1}{e}.
\]
This is best possible: no $q>1/e$ quantile share can be universally feasible.
\end{theorem}

More generally, the upper bound extends to all thinning parameters.

\begin{proposition}\label{prop:upper-bound}
For every $c\in(0,1]$, the $c$-thinned $q$-quantile share is not universally feasible for any $q>e^{-c}$.
\end{proposition}

\begin{proof}
Fix $c\in(0,1]$ and $q>e^{-c}$. Choose $n$ sufficiently large so that
\[
\left(1-\frac{c}{n}\right)^{n-1}<q.
\]
Let $m=n-1$, and for every agent let
\[
v(T):=\ind\{T\ne \emptyset\}.
\]
Then
\[
\PP\bigl[v(X_i^{(c)})=0\bigr]=\left(1-\frac{c}{n}\right)^{n-1}<q,
\]
so $\tau_q^{(c)}(v,n)=1$ for each agent. But with only $n-1$ goods, no allocation can give every one of the $n$ agents a nonempty bundle. 
\end{proof}

\subsection{Why thinning helps.}
\label{sec:thinning-motivation}
Given the large literature on constant-factor approximations to the proportional share and MMS, a natural first thought would be to weaken the quantile share in the same way, namely by asking for a positive fraction $\alpha \tau_q(v,n)$ of the original $q$-quantile share. This does not actually weaken the problem on the hard $0/1$ core. As the following observation shows, for monotone $0/1$-valuations multiplying a quantile share by any positive constant leaves the induced allocation problem unchanged. 

\begin{observation}\label{prop:alpha-not-easier}
Fix $\alpha\in(0,1]$. On the class $\Vzeroone$, the allocation problem defined by the benchmark $\alpha\tau_q$ is equivalent to the allocation problem defined by $\tau_q$.
\end{observation}

\begin{proof}
Let $u\in \Vzeroone$. Since $u$ takes only the values $0$ and $1$, its $q$-quantile share also belongs to $\{0,1\}$. If $\tau_q(u,n)=0$, then both benchmarks are trivial. If $\tau_q(u,n)=1$, then for every bundle $T$ we have
\[
u(T)\ge \alpha\tau_q(u,n)
\iff u(T)\ge \alpha
\iff u(T)=1
\iff u(T)\ge \tau_q(u,n). \qedhere
\]
\end{proof}
Thus, if one wants a meaningful relaxation of the quantile benchmark, the constant has to be inserted at the level of the \emph{distribution} generating the benchmark bundle, rather than at the level of the resulting share value. This is exactly what thinning does.

Once the constant is placed in the distribution, the reduction of~\cite{BFHN} to the rainbow Erd\H{o}s matching conjecture changes in exactly the right way. Roughly, the proof in \cite{BFHN} reduces universal feasibility of the original quantile share to a rainbow matching statement in the near-perfect matching regime of the rainbow Erd\H{o}s matching problem. For the $c$-thinned benchmark, a similar reduction leads to a regime which is linearly away from the near-perfect threshold. Existing rainbow-matching theorems apply there unconditionally, which yields \cref{thm:main}; see \cref{sec:reduction,sec:rainbow} for the details. A byproduct of this is that \emph{any} progress on the rainbow Erd\H{o}s matching problem will automatically improve the constant $c$ in \cref{thm:main,thm:quantitative}. 

The same viewpoint also leads to the optimal version of the original quantile-share result. Babichenko et al.~\cite{BFHN} obtain the feasible quantile $1/(2e)$ from a median-type estimate: in the unthinned benchmark, the relevant event is that the random bundle has size at most about its mean, which only carries probability about $1/2$. By first passing to a $(1-\varepsilon)$-thinned benchmark, one moves into a genuine lower-tail regime and can replace that median bound by a Chernoff estimate with probability $1-o_M(1)$ as the blow-up parameter $M\to\infty$ (with $n$ and $\varepsilon$ fixed). A local-constancy argument then transfers the conclusion back to the unthinned benchmark, giving the optimal conditional value $1/e$ in \cref{thm:oneovere}.

It is worth emphasizing why this avoids the obstruction identified in~\cite[Section~5.1]{BFHN}. That obstruction concerns methods that attempt to improve the quantile level while still giving agents approximately equal-sized bundles of the original goods. We instead add many dummy goods and run the rainbow-matching argument on a $k$-uniform layer of the enlarged ground set $[M]$, with $k\approx \gamma M/n$ for some $\gamma<1$. The resulting sets are equal-sized only in the padded universe. After deleting the dummy goods, the real bundles may have very different sizes and need not resemble bundles of size $m/n$. Taking $M$ much larger than $m$ makes the equal-size constraint live almost entirely in the dummy coordinates, while the real goods are governed by the thinned lower-tail distribution.

\subsection{Comparison to RMMS}
\label{sec:rmms}
Recently, Feige~\cite{FeigeRMMS} introduced the \emph{residual maximin share}
(RMMS), which is universally feasible and self-maximizing. We recall the
definition. A threshold $t\ge 0$ is \emph{residually self-feasible} for $(v,n)$
if, for every $k\in\{0,\dots,n-1\}$ and every collection of pairwise disjoint
bundles $B_1,\dots,B_k$ satisfying $v(B_r)<t$ for all $r$, the remaining goods
\[
[m]\setminus \bigcup_{r=1}^k B_r
\]
can be partitioned into $n-k$ bundles each of value at least $t$. Then
\[
\operatorname{RMMS}(v,n):=
\sup\{t\ge 0:\ t \text{ is residually self-feasible for }(v,n)\}.
\]

In \cref{sec:rmms-incomp}, we show that RMMS and quantile-based shares are incomparable: depending on the valuation, one share may be greater than the other. In particular, \cref{example:complementarity} gives a simple example of a 0/1-valuation in the setting of complementary goods where the RMMS share achieves the minimum possible value of $0$, but the $c$-thinned $e^{-c}$-quantile share achieves the maximum possible value of $1$.  

Apart from being able to handle complementarity, we believe that quantile-based shares are generally attractive because they compare an allocation to an
explicit random benchmark. This makes the guarantee interpretable: rather than
asking for a value defined through all residual partitions of the ground set, the
agent asks to receive a bundle at least as good as what she would obtain from a
specified lottery (except with some controlled lower-tail probability).

The benchmark lottery involved in the definition of $c$-thinned quantile shares is also natural under stochastic supply. As a concrete example, suppose that specialized medical kits are to be distributed among clinics,
but each kit independently arrives on time with probability $c$, and the kits
that do arrive are assigned by a random lottery. Under this policy,
each kit belongs to clinic $i$'s benchmark bundle independently with probability
$c/n$, exactly matching the $c$-thinned random-bundle model. Therefore, if the $c$-thinned $e^{-c}$-quantile share were feasible with $c \approx 0.7$, then in a setting where kits arrive on time with probability $\approx 70\%$, there would be an allocation which guaranteed that each agent receives at least the median value of the benchmark distribution.

Finally, another advantage of quantile-based shares is computability. Essentially by definition, one can approximate quantile-based shares using Monte Carlo methods. On the other hand, Feige shows that computing RMMS
is weakly NP-hard already for additive valuations~\cite{FeigeRMMS}.

\subsection{Organization}
\cref{sec:comparisons} provides examples comparing the thinned quantile share with the original quantile share, and with RMMS. \cref{sec:reduction} gives the reduction from fair division to cross-dependent families. \cref{sec:rainbow} recalls the required inputs from extremal set theory and completes the proofs of \cref{thm:main,thm:oneovere}.

\subsection{Acknowledgements}
V.J.~is partially supported by NSF grant DMS-2237646. C.M.~is partially supported by a Simons Dissertation Fellowship.

\section{Examples and comparison with other shares}
\label{sec:comparisons}

\subsection{Comparison with the ordinary quantile share}
Ordinary quantile shares have several appealing properties~
\cite{BFHN}: they are ordinal, invariant under renaming of goods, and self-maximizing (in the sense of Babaioff and Feige~\cite{BF}). They are also monotone and $1$-Lipschitz in the valuation. All of these properties extend to thinned quantile shares, with the same proofs.

\medskip

The next proposition compares the $c$-thinned $e^{-c}$-quantile share for different values of $c \in (0,1]$. \emph{A priori}, it is not clear these shares should be comparable: for $0 < c < c' \leq 1$, the $c$-thinned distribution is a further thinning of the $c'$-thinned distribution, but on the other hand, we are allowed to take a larger quantile for it. However, as the next proposition shows, the $c$-thinned $e^{-c}$-quantile share is nondecreasing in $c$.  

\begin{proposition}
\label{prop:monotonicity-thinning}
Fix $n\ge 2$, let $v:2^{[m]}\to\mathbb R_{\ge 0}$ be monotone, and let
$0<c\le c'\le 1$. Then, for every $q\in(0,1)$,
\[
\tau_q^{(c)}(v,n)\le \tau_{q^{c'/c}}^{(c')}(v,n).
\]
In particular,
\[
\tau_{e^{-c}}^{(c)}(v,n)\le \tau_{e^{-c'}}^{(c')}(v,n).
\]
\end{proposition}

\begin{proof}

We use the standard down-set inequality
\begin{equation}\label{eq:downset-concavity}
\mu_{\alpha p}(D)\ge \mu_p(D)^\alpha ,
\end{equation}
valid for every down-set $D\subseteq 2^{[m]}$, every $p\in[0,1]$, and every
$\alpha\in[0,1]$, where $\mu_p$ denotes the product measure on $2^{[m]}$ with
inclusion probability $p$. We present a proof of this inequality at the end of the proof for the reader's convenience. 

For $t\in\mathbb R_{\ge 0}$, set
\[
D_t:=\{S\subseteq[m]:v(S)\le t\}.
\]
Since $v$ is monotone, $D_t$ is a down-set. Apply
\eqref{eq:downset-concavity} with
\[
p=\frac{c'}{n},
\qquad
\alpha=\frac{c}{c'}.
\]
Then
\[
\PP\bigl[v(X^{(c)})\le t\bigr]
=
\mu_{c/n}(D_t)
\ge
\mu_{c'/n}(D_t)^{c/c'}
=
\PP\bigl[v(X^{(c')})\le t\bigr]^{c/c'}.
\]
Hence, whenever
\[
\PP\bigl[v(X^{(c)})\le t\bigr]<q,
\]
we also have
\[
\PP\bigl[v(X^{(c')})\le t\bigr]<q^{c'/c}.
\]
Let $t=\tau_{q^{c'/c}}^{(c')}(v,n)$. Then $\mathbb P[v(X^{(c')})\le t]\ge q^{c'/c}$. By the preceding inequality, $\mathbb P[v(X^{(c)})\le t]\ge q$. Hence $\tau_q^{(c)}(v,n)\le t=\tau_{q^{c'/c}}^{(c')}(v,n)$.
The final assertion follows by substituting $q=e^{-c}$.

\medskip

It remains to prove \cref{eq:downset-concavity}. It is enough to prove that the function
\[
p\longmapsto \mu_p(D)^{1/p}
\]
is nonincreasing on $(0,1]$. Write
\[
F(p):=\mu_p(D).
\]
We prove the differential inequality
\[
-pF'(p)\ge F(p)\log\frac1{F(p)}.
\]
This implies
\[
\frac{d}{dp}\left(\frac{\log F(p)}{p}\right)
=
\frac{pF'(p)/F(p)-\log F(p)}{p^2}
\le 0,
\]
and therefore $F(q)\ge F(p)^{q/p}$ whenever $0<q\le p$. Taking
$q=\alpha p$ gives the claim.

It remains to prove the differential inequality. We prove, by induction on
$m$, that for every down-set $D\subseteq 2^{[m]}$ and every $p\in(0,1)$,
\begin{equation}\label{eq:downset-differential}
    -pF'(p)\ge F(p)\log\frac1{F(p)},
    \qquad F(p):=\mu_p(D).
\end{equation}
Here and below we use the convention that $0\log(1/0)=0$.

The cases $D=\emptyset$ and $D=2^{[m]}$ are immediate, so assume
$0<F(p)<1$. The case $m=0$ is then vacuous. For $m\ge 1$, decompose $D$
according to whether the last coordinate is present:
\[
D_0:=\{S\subseteq [m-1]:S\in D\},
\qquad
D_1:=\{S\subseteq [m-1]:S\cup\{m\}\in D\}.
\]
Both $D_0$ and $D_1$ are down-sets in $2^{[m-1]}$, and since $D$ is a
down-set we have $D_1\subseteq D_0$. Let
\[
A(p):=\mu_p(D_0),
\qquad
B(p):=\mu_p(D_1).
\]
Then $0\le B(p)\le A(p)\le 1$ and
\[
F(p)=(1-p)A(p)+pB(p).
\]
Differentiating this identity gives
\[
-F'(p)=A(p)-B(p)+(1-p)(-A'(p))+p(-B'(p)).
\]
By the induction hypothesis applied to $D_0$ and $D_1$,
\[
-pA'(p)\ge A(p)\log\frac1{A(p)},
\qquad
-pB'(p)\ge B(p)\log\frac1{B(p)}.
\]
Thus
\begin{equation}\label{eq:inductive-lower-bound}
p(-F'(p))
\ge
p(A-B)+(1-p)A\log\frac1A+pB\log\frac1B.
\end{equation}

We now compare the right-hand side of \eqref{eq:inductive-lower-bound} with
$F\log(1/F)$. If $A=0$, then also $B=0$ and $F=0$, contrary to our 
assumption. Hence $A>0$. Write
\[
B=xA
\qquad\text{with }x\in[0,1],
\]
and set
\[
h:=1-p+px.
\]
Then $F=Ah$. Substituting $B=xA$ into the right-hand side of
\eqref{eq:inductive-lower-bound}, we obtain
\[
p(A-B)+(1-p)A\log\frac1A+pB\log\frac1B
=
A\left[
h\log\frac1A+p(1-x)+px\log\frac1x
\right],
\]
where the term $x\log(1/x)$ is interpreted as $0$ when $x=0$. On the other
hand,
\[
F\log\frac1F
=
Ah\log\frac1{Ah}
=
A\left[
h\log\frac1A+h\log\frac1h
\right].
\]
Therefore \eqref{eq:downset-differential} will follow once we prove
\[
p(1-x)+px\log\frac1x\ge h\log\frac1h.
\]
But \(h=1-p(1-x)\). Since \(0<h\le 1\), the elementary inequality
\[
y\log\frac1y\le 1-y
\qquad (0<y\le 1)
\]
gives
\[
h\log\frac1h\le 1-h=p(1-x).
\]
This proves the desired differential inequality.
\end{proof}

The previous proposition shows that the $c$-thinned $e^{-c}$-quantile share is at most the original $e^{-1}$-quantile share. The next two examples compare these shares on some natural valuations.

\begin{example}\label{ex:identical-goods}
Let $v(T)=|T|$. Fix $n\ge 2$, $q\in(0,1)$, and $c\in(0,1]$. Then, as
$m\to\infty$,
\[
\tau_q(v,n)=\frac{m}{n}+o(m)
\qquad\text{and}\qquad
\tau_q^{(c)}(v,n)=\frac{cm}{n}+o(m).
\]
Consequently,
\[
\frac{\tau_q^{(c)}(v,n)}{\tau_q(v,n)}\to c.
\]
\end{example}

\begin{proof}
By definition, $|X^{(c)}|\sim\Bin(m,c/n)$. For every fixed $\delta>0$, Hoeffding's
inequality gives
\[
\PP\left[\left||X^{(c)}|-\frac{cm}{n}\right|>\delta m\right]
\le 2e^{-2\delta^2m}.
\]
Thus, for all sufficiently large $m$,
\[
\PP\left[|X^{(c)}|\le \frac{cm}{n}-\delta m\right]<q
\le
\PP\left[|X^{(c)}|\le \frac{cm}{n}+\delta m\right].
\]
By the definition of the left quantile,
\[
\tau_q^{(c)}(v,n)
\in
\left[\frac{cm}{n}-\delta m,\frac{cm}{n}+\delta m\right]
\]
for all sufficiently large $m$. Since $\delta>0$ was arbitrary, this proves
\[
\tau_q^{(c)}(v,n)=\frac{cm}{n}+o(m).
\]
Taking $c=1$ gives the ordinary quantile-share asymptotic.
\end{proof}

\begin{example}\label{ex:threshold}
For an integer threshold $T\ge 0$, define
\[
u_T(S):=\ind\{|S|\ge T\}.
\]
Fix $n\ge 2$, $q\in(0,1)$, and $c\in(0,1]$, and let $T=T_m$ depend on $m$.
Write
\[
Q_c(m):=\tau_q^{(c)}(u_{T_m},n),
\qquad
Q_1(m):=\tau_q(u_{T_m},n).
\]
Then, as $m\to\infty$, the following hold:
\[
\begin{array}{rcl}
T_m\le cm/n-\omega(\sqrt m)
&\Longrightarrow&
Q_c(m)=Q_1(m)=1,\\[1mm]
T_m\ge m/n+\omega(\sqrt m)
&\Longrightarrow&
Q_c(m)=Q_1(m)=0,\\[1mm]
cm/n+\omega(\sqrt m)\le T_m\le m/n-\omega(\sqrt m),\ c<1
&\Longrightarrow&
Q_c(m)=0,\quad Q_1(m)=1.
\end{array}
\]
\end{example}

\begin{remark}
Thus threshold valuations show both behaviours. Away from the window between
the $c$-thinned mean $cm/n$ and the ordinary mean $m/n$, thinning leaves the
benchmark unchanged. Inside that window, however, thinning can change the
quantile share from $1$ to $0$.
\end{remark}

\begin{proof}
By definition, $|X^{(a)}|\sim\Bin(m,a/n)$ and
\[
u_{T_m}(X^{(a)})=1
\qquad\Longleftrightarrow\qquad
|X^{(a)}|\ge T_m.
\]
By binomial concentration, if
\[
T_m\le \frac{am}{n}-\omega(\sqrt m),
\]
then
\[
\PP[|X^{(a)}|\ge T_m]\to 1,
\]
whereas if
\[
T_m\ge \frac{am}{n}+\omega(\sqrt m),
\]
then
\[
\PP[|X^{(a)}|\ge T_m]\to 0.
\]
For a $0/1$-valued random variable, the left $q$-quantile is $1$ exactly when
the probability of value $0$ is strictly less than $q$; otherwise it is $0$.
Applying this criterion with $a=c$ and $a=1$ proves all three cases.
\end{proof}

\subsection{Incomparability with RMMS}
\label{sec:rmms-incomp}
Recall Feige's \emph{residual maximin share}~\cite{FeigeRMMS} from \cref{sec:rmms}. A threshold $t\ge 0$ is \emph{residually self-feasible} for $(v,n)$ if for every $k\in\{0,\dots,n-1\}$ and every collection of pairwise disjoint bundles $B_1,\dots,B_k$ satisfying $v(B_r)<t$ for all $r$, the remaining goods
\[
[m]\setminus \bigcup_{r=1}^k B_r
\]
can be partitioned into $n-k$ bundles each of value at least $t$. Then
\[
\operatorname{RMMS}(v,n):=\sup\{t\ge 0:\ t \text{ is residually self-feasible for }(v,n)\}.
\]
We present two examples showing that thinned quantile shares are, in general, incomparable with RMMS.

\begin{example}\label{prop:rmms-identical}
Let $v(T)=|T|$. Then
\[
\RMMS(v,n)=\MMS(v,n)=\left\lfloor\frac{m}{n}\right\rfloor.
\]
Consequently, for every fixed $c\in(0,1)$ and every fixed $q\in(0,1)$,
\[
\frac{\tau_q^{(c)}(v,n)}{\RMMS(v,n)}\to c
\qquad\text{as }m\to\infty.
\]
\end{example}

\begin{proof}
For identical unit goods, the optimal partition for the maximin share is the balanced partition, so
 $\MMS(v,n)=\lfloor m/n\rfloor$. Since RMMS is never larger than MMS by definition, it suffices to show that $\RMMS(v,n)\ge \lfloor m/n\rfloor$. Let $t=\lfloor m/n\rfloor$ and suppose we remove $r$ bundles of size at most $t-1$. Then, at least $m-r(t-1)$ goods remain. Therefore the remaining goods can be partitioned into $n-r$ bundles of size at least $t$. Hence, $\RMMS(v,n)\ge t$, as claimed. The asymptotic comparison with $\tau_q^{(c)}(v,n)$ now follows from \cref{ex:identical-goods}.
\end{proof}

The preceding example shows that for identical unit goods, $\RMMS$ exceeds the thinned quantile share asymptotically by a constant factor. On the other hand, as the example below shows, RMMS can be arbitrarily bad compared to quantile shares in the presence of complementarity. 

\begin{example}\label{example:complementarity}
Let $n\ge 2$ and $c\in(0,1]$. Choose an integer $a\ge 1$ so that 
\[
\bigl(1-(1-c/n)^a\bigr)^2>1-e^{-c};
\]
this is possible for sufficiently large $a$ since $c > 0$. Let the goods be two disjoint blocks
\[
R=\{r_1,\dots,r_a\},
\qquad
B=\{b_1,\dots,b_a\},
\]
and define
\[
v(S):=\mathbf 1_{\{S\cap R\neq\emptyset\ \text{and}\ S\cap B\neq\emptyset\}}.
\]
In words, an agent is satisfied if she gets at least one red item and at least one blue item. Then,
\[
\operatorname{RMMS}(v,n)=0
\qquad\text{but}\qquad
\tau_{e^{-c}}^{(c)}(v,n)=1.
\]

\end{example}

\begin{proof}

We first show that $\operatorname{RMMS}(v,n)=0$. Suppose for contradiction that $\RMMS(v,n)>0$. Then there exists a residually self-feasible threshold $t>0$. Since the bundle $R$ has value $0<t$, it is an admissible bundle to remove in the definition of residual self-feasibility. After removing $R$, only goods from $B$ remain, and
therefore every remaining bundle has value $0$. Hence the remaining goods
cannot be partitioned into $n-1$ bundles of value at least $t$, which gives the desired contradiction. 

Now let $X^{(c)}$ be the $c$-thinned random bundle. Since each good is included
independently with probability $p=c/n$, and since $R$ and $B$ are disjoint,
\[
\PP\bigl[v(X^{(c)})=1\bigr]
=
\PP[X^{(c)}\cap R\neq\emptyset]\PP[X^{(c)}\cap B\neq\emptyset]
=
\bigl(1-(1-p)^a\bigr)^2
>
1-e^{-c},
\]
so that
\[\tau_{e^{-c}}^{(c)}(v,n)=1. \qedhere\]
\end{proof}

\section{Reduction}
\label{sec:reduction}

The goal of this section is to prove \cref{prop:transfer}, which reduces our fair-division problem to a statement in extremal set theory. Compared to the main result of \cite{BFHN}, there are two crucial differences: first, we allow for an arbitrary thinning parameter $c \in (0,1)$ and second, we obtain the optimal value of the quantile $q = e^{-c}$.

\begin{definition}
Let $\mathcal F_1,\dots,\mathcal F_n$ be families of subsets of the same ground set $[M]$. We say that they are \emph{cross-dependent} if one cannot choose sets
\[
F_1\in\mathcal F_1,\ \dots,\ F_n\in\mathcal F_n
\]
that are pairwise disjoint. 
\end{definition}

In the discussion below, we will be interested in cross-dependent families where each $\mc{F}_i \subseteq \binom{[M]}{k}$. Note that if $\mc{F}_1,\dots,\mc{F}_n$ are cross-dependent and $\mc{G}_1 \subseteq \mc{F}_1,\dots, \mc{G}_n \subseteq \mc{F}_n$, then $\mc{G}_1,\dots, \mc{G}_n$ are also cross-dependent. Hence, the natural extremal problem is to find cross-dependent families where the components $\mc{F}_i$ are as large as possible. 

There are two trivial ways of constructing cross-dependent families with large $|\mc{F}_i|$. First, we can choose each $\mc{F}_i$ to be the collection of all $k$-subsets of a distinguished set of $kn-1$ elements in $[M]$. Second, we can require each $\mc{F}_i$ to consist of all $k$-subsets of $[M]$ which have non-empty intersection with a distinguished set of $n-1$ elements in $[M]$. The following conjecture, which is perhaps the most famous open problem in extremal set theory, asserts that one cannot do better than these trivial constructions.

\begin{conjecture}[Rainbow Erd\H{o}s matching conjecture \cite{aharoni2017rainbow,huang2012size}]\label{conj:rainbow}
Let $n,k,M\in\N$ with $M\ge nk$. If
\[
\mathcal F_1,\dots,\mathcal F_n\subseteq \binom{[M]}{k}
\]
are cross-dependent, then
\[
\min_{j\in[n]}|\mathcal F_j|
\le
\max\left\{
\binom{M}{k}-\binom{M-n+1}{k},\,
\binom{kn-1}{k}
\right\}.
\]
\end{conjecture}

We will return to the rainbow Erd\H{o}s matching conjecture in \cref{sec:rainbow}. For now, we proceed to our reduction. 

\begin{proposition}\label{prop:transfer}
Fix $n \geq 2$.  Suppose there exists $C>1$ such that for all sufficiently large integers $k$ and $M \geq Ckn$, every cross-dependent collection

\[
\mc F_1,\dots,\mc F_n\subseteq \binom{[M]}{k}
\]
satisfies
\[
\min_{i\in[n]}|\mc F_i|
\le
\binom{M}{k}-\binom{M-n+1}{k}.
\]

Then, for $c = 1/C$, the $c$-thinned $q_c$-quantile share is feasible for every profile of $n$ monotone valuations, where
\[
q_c= q_{c,n} := \left(1-\frac cn\right)^{n-1} > e^{-c}.
\]
\end{proposition}

The remainder of this section is devoted to the proof of \cref{prop:transfer}. As in \cite{BFHN}, we first pass from general monotone valuations to monotone $0/1$ valuations. 

\begin{lemma}\label{lem:01-reduction}
Fix $c\in(0,1]$ and $q \in (0,1)$. Let $v_1,\dots,v_n\in\V$, and write
\[
\tau_i:=\tau_q^{(c)}(v_i,n).
\]
For every agent $i$ with $\tau_i>0$, there exists a valuation $u_i\in\Vzeroone$ such that
\begin{enumerate}[label=\textup{(\roman*)}]
    \item $\tau_q^{(c)}(u_i,n)=1$, and
    \item every bundle $T$ with $u_i(T)=1$ satisfies $v_i(T)\ge \tau_i$.
\end{enumerate}
\end{lemma}

\begin{proof}

Fix some agent $i$ with $\tau_i > 0$ and define
\[ u_i(T) = \mbf1\{ v_i(T) \geq \tau_i \}. \]
Then $u_i$ is a monotone $0/1$-valuation and (ii) is immediate. It remains to prove (i). By the minimality of $\tau_i$ in the definition of the left $q$-quantile, and since $\tau_i>0$, we have \[ \mathbb P\bigl[v_i(X^{(c)})<\tau_i\bigr]<q. \] Equivalently, \[ \mathbb P\bigl[u_i(X^{(c)})=0\bigr]<q. \] For a $0/1$-valued random variable, the left $q$-quantile is equal to $1$ exactly when the probability of value $0$ is strictly less than $q$. Hence $\tau_q^{(c)}(u_i,n)=1$.
\end{proof}

By padding with ``dummy goods'', we can also reduce to the setting where the number of goods is sufficiently large compared to the number of agents. 

\begin{lemma}\label{lem:padding}
Fix $c\in(0,1]$, $q \in (0,1)$, and let $v\in \V$ be a valuation on $[m]$. Let $M\ge m$, and extend $v$ to a valuation $\widetilde v$ on $[M]$ by
\[
\widetilde v(T):=v(T\cap [m]).
\]
Then
\[
\tau_q^{(c)}(\widetilde v,n)=\tau_q^{(c)}(v,n).
\]
\end{lemma}

\begin{proof}
Under the $c$-thinned process on $[M]$, the random set $X^{(c)}\cap [m]$ has exactly the same distribution as the $c$-thinned random bundle on $[m]$. Since $\widetilde v(T)=v(T\cap [m])$, the random variables $\widetilde v(X^{(c)})$ and $v(X^{(c)})$ are identically distributed. Hence their quantiles are identical.
\end{proof}

Finally, we will use the following classical shadow estimate.

\begin{theorem}[Kruskal--Katona, Lov\'asz form \cite{lovasz2007combinatorial}]\label{thm:KK}
Let $\mathcal G_k\subseteq \binom{[M]}{k}$, and for $0\le t\le k$ define its $t$-shadow by
\[
\partial_t\mathcal G_k:=\left\{S\in\binom{[M]}{t}:\exists T \in \mathcal G_k\text{ with }S\subseteq T\right\}.
\]
If $|\mathcal G_k|\ge \binom{x}{k}$ for some real $x\ge k$, then
\[
|\partial_t\mathcal G_k|\ge \binom{x}{t} \qquad\text{for every }0\le t\le k. 
\]

\end{theorem}

We now have all the ingredients needed to prove \cref{prop:transfer}.

\begin{proof}[Proof of \cref{prop:transfer}]
Let $n\geq 2$ be the number of agents and let 
\[c = 1/C, \qquad q_c = (1-c/n)^{n-1}.\] 
Suppose, for contradiction, that there is a profile $v_1,\dots,v_n$ of monotone valuations on $[m]$ which admits no allocation giving every agent her $c$-thinned $q_c$-quantile share.
Let 
\[\tau_i:=\tau_{q_c}^{(c)}(v_i,n),
\]
and let $I\subseteq[n]$ be the set of active agents, namely those with $\tau_i>0$. If $|I|\le 1$, then a fair allocation is immediate, so we may assume that $|I|\ge 2$. Apply \cref{lem:01-reduction} to every active agent. Thus, for each $i\in I$, there is a monotone $0/1$ valuation $u_i$ such that
\[
\tau_{q_c}^{(c)}(u_i,n)=1,
\]
and such that any bundle $T$ with $u_i(T)=1$ satisfies $v_i(T)\ge \tau_i$. Therefore it is enough to find pairwise disjoint bundles giving every active agent value $1$ under her corresponding $u_i$.

For a collection of bundles $\mc{A} \subseteq 2^{[m]}$ and $a \in [0,1]$, define the weight
\[ w(\mc{A},a) := \mb{P}[X^{(a)} \in \mc{A}] =  \sum_{S \in \mc{A}} \mb{P}[X^{(a)} = S].\]
For $i \in I$, let $\mc{R}_i$ denote the collection of bundles that agent $i$ rejects
\[\mc{R}_i := \{S\subseteq [m]: u_i(S) = 0\}.\]
Since $u_i$ is $0/1$-valued and $\tau_{q_c}^{(c)}(u_i,n) = 1$, it follows that
\[w(\mc{R}_i, c) < q_c.\]

Since $a\mapsto w(\mathcal R_i,a)-q_a$ is continuous and is negative at $a=c$, there exists $\varepsilon_i>0$ such that \[ w(\mathcal R_i,a)<q_a \] for every $a\in[c-\varepsilon_i,c]$. Let $\varepsilon:=\min\{c/2,\min_{i\in I}\varepsilon_i\}$ and set $\gamma:=c-\varepsilon$. Then $\gamma\in(0,c)$ and $w(\mathcal R_i,\gamma)<q_\gamma$, and therefore $\tau_{q_\gamma}^{(\gamma)}(u_i,n)=1$, for every $i\in I$.

We now pad the instance with dummy goods. Throughout, assume that $n,m,\gamma, C$ are fixed and let $M\ge m$ be sufficiently large for various inequalities below to hold. Extend each $u_i$ to $[M]$ by ignoring the dummy goods, and continue to denote the extended valuations by $u_i$. By \cref{lem:padding}, the active agents still satisfy
\[
\tau_{q_\gamma}^{(\gamma)}(u_i,n)=1.
\]
Moreover, if the padded instance admits an allocation giving every active agent value $1$ under $u_i$, then intersecting the allocated bundles with $[m]$ gives a fair allocation for the original profile. Thus the padded instance is still a counterexample.

Let
\[ k := \left\lfloor \frac \gamma n \paren{M + M^{2/3}} \right\rfloor. \]
The reason for this choice will become clear below.  
Since $\gamma < c = 1/C$, we have $C \gamma < 1$. Therefore, for all sufficiently large $M$,

\[ Ckn \leq C \gamma \paren{M + M^{2/3}} \leq M\]
which will allow us to invoke the hypothesis in the statement of \cref{prop:transfer}.

For each active agent $i \in I$, define the $k$-uniform family
\[
\mc F_i:=\Bigl\{T\in \binom{[M]}{k}:u_i(T)=1\Bigr\}.
\]
For a non-active agent $j \in [n]\setminus I$, define $\mc{F}_j := \binom{[M]}{k}$. We claim that the families $\mc{F}_1,\dots, \mc{F}_n$ are cross-dependent. Indeed, if there were pairwise disjoint sets
$T_i\in \mc F_i$ for all $i\in[n]$, then the active agents $i$ would receive pairwise disjoint bundles $T_i$ with $u_i(T_i) = 1$, and by monotonicity of the valuations, one could distribute the
remaining goods arbitrarily and obtain an allocation giving every active agent value $1$. This contradicts the choice of the counterexample.

Since $\mc{F}_1,\dots, \mc{F}_n$ are cross-dependent, it follows by assumption that there exists an agent $i_* \in [n]$ such that
\[
|\mc F_{i_*}|\le \binom{M}{k}-\binom{M-n+1}{k}.
\]

For all sufficiently large $M$, the right hand side is smaller than $\binom{M}{k}$, so that the agent $i_*$ must be active. By relabeling if necessary, assume that $i_*=1 \in I$. 

 For each $0\le t\le k$, define
\[
\mc G_{t}:=\Bigl\{T\in \binom{[M]}{t}:u_1(T)=0\Bigr\}.
\]

In particular,
\[
\mc G_k=\binom{[M]}{k}\setminus \mc F_1,
\]
so that
\[|\mc{G}_k| \geq \binom{M-n+1}{k}.\]
Since $u_1$ is monotone, every subset of a rejected bundle is rejected. Hence,
\[
\partial_t\mc G_k\subseteq \mc G_t.
\]
By \cref{thm:KK},
\[
|\mc G_t|\ge |\partial_t\mc G_k|\ge \binom{M-n+1}{t}
\qquad\text{for every }0\le t\le k.
\]

We now estimate the rejection probability of agent $1$ under the $\gamma$-thinned random bundle distribution:
\begin{align*}
\mb{P}\bigl[u_1(X_1^{(\gamma)})=0\bigr]
& \geq \sum_{t = 0}^k \mb{P}\bigl[|X_1^{(\gamma)}| = t\bigr] \mb{P}\bigl[u_1(X_1^{(\gamma)})=0 \mid |X_1^{(\gamma)}|=t\bigr] \\
& = \sum_{t=0}^k \binom Mt \paren{\frac \gamma n}^t \paren{1 - \frac \gamma n}^{M-t} \cdot \frac{|\mc G_t|}{\binom{M}{t}} \\
& \geq \sum_{t=0}^k \binom{M-n+1}{t} \paren{\frac \gamma n}^t \paren{1 - \frac \gamma n}^{M-t} \\
& = \paren{1 - \frac \gamma n}^{n-1} \mb{P}\bigl[ \Bin(M-n+1,\gamma /n) \leq k \bigr].
\end{align*}

By the choice of $k$,
\[k - \frac{\gamma}{n}(M-n+1) = \Omega_{n,\gamma}(M^{2/3}).\]
By a standard Chernoff bound, \[ \mathbb P\bigl[\operatorname{Bin}(M-n+1,\gamma/n)\le k\bigr] \ge 1-\exp\bigl(-\Omega_{\gamma,n}(M^{1/3})\bigr)=1-o_M(1), \]
as $M \to \infty$ holding all other parameters constant.  

Since 
\[\mb{P}[u_1(X_1^{(\gamma)}) = 0] = w(\mc{R}_1, \gamma) < q_\gamma = \paren{1-\frac\gamma n}^{n-1},\]
by our choice of $\gamma$,
we may choose $M$ sufficiently large so that
\[
\mb{P}[u_1(X_1^{(\gamma)})=0] \geq \left(1-\frac{\gamma}{n}\right)^{n-1}
\PP\left[
\Bin\left(M-n+1,\frac{\gamma}{n}\right)\le k
\right]
>
\mb{P}[u_1(X_1^{(\gamma)}) = 0],
\]
which is a contradiction. \qedhere

\end{proof}

\section{Proofs of the main theorems}
\label{sec:rainbow}

For progress on the Erd\H{o}s matching conjecture and its rainbow variant, we refer the reader to the recent survey~\cite{RainbowSurvey}. For \cref{thm:main}, we will only need the following result, which combines two unconditional theorems. The first part is due to Frankl--Kupavskii \cite{FK} and the second part is due to Kupavskii \cite{Kup}.

\begin{theorem}[Frankl--Kupavskii~\cite{FK}, Kupavskii~\cite{Kup}]
\label{thm:rainbow-matching}
Let $n,k, M \in \mb{N}$ with $M \geq kn$. Suppose
\[
\mathcal F_1,\dots,\mathcal F_n\subseteq \binom{[M]}{k}
\]
are cross-dependent. Then,
\[
\min_{j\in[n]}|\mathcal F_j|
\le
\binom{M}{k}-\binom{M-n+1}{k}
\]
under one of the following two conditions:
\begin{enumerate}
\item \label{thm:FK} $M \geq 12 nk \log (e^2 n)$; 
\item \label{thm:Kup} $M > 3enk$ and $n > 10^7 + 1$. 
\end{enumerate}
\end{theorem}

We now state and prove a quantitative version of \cref{thm:main}.

\begin{theorem}\label{thm:quantitative}
The following hold unconditionally:
\begin{enumerate}[label=\textup{(\roman*)}]
    \item If $n\geq 2$ and
    \[
    0<c<\frac{1}{12\log(e^2n)},
    \]
    then for every $q\le (1-c/n)^{n-1}$, the $c$-thinned $q$-quantile share is feasible for every profile of $n$ monotone valuations.
    \item If $n\ge 10^7+2$ and
    \[
    0<c<\frac{1}{3e},
    \]
    then the same conclusion holds.
    \item Consequently, if $n\geq 2$ and 
    \[
    0<c<\frac{1}{12\log\bigl(e^2(10^7+1)\bigr)},
    \]
    then for every $q\le e^{-c}$ the $c$-thinned $q$-quantile share is universally feasible.
\end{enumerate}
\end{theorem}

\begin{proof}[Proof of \cref{thm:quantitative}]
For (i), fix $0 < c<1/(12\log(e^2n))$ and apply \cref{prop:transfer} with $C=1/c$. Since $C>12\log(e^2n)$, the Frankl--Kupavskii theorem applies whenever $M\ge Ckn$. This gives feasibility at $q_c=(1-c/n)^{n-1}$, and hence at every $q\le q_c$.

For (ii), fix $c<1/(3e)$ and apply \cref{prop:transfer} with $C=1/c$. Since $C>3e$, Kupavskii's theorem applies whenever $M\ge Ckn$, and the same monotonicity in $q$ gives the result for all $q\le (1-c/n)^{n-1}$.

For (iii), if $n\le 10^7+1$, then the assumed bound on $c$ implies the hypothesis of (i). If $n>10^7+1$, the same bound implies $c<1/(3e)$, so (ii) applies. In both cases, $(1-c/n)^{n-1}>e^{-c}$, so feasibility for all $q\le e^{-c}$ follows.
\end{proof}

Finally, we prove \cref{thm:oneovere}.

\begin{proof}[Proof of \cref{thm:oneovere}]
We repeat the proof of \cref{prop:transfer} with $c = 1$ and $q_c = (1-c/n)^{n-1}$. Namely, suppose for contradiction that there is a profile $v_1,\dots, v_n$ of monotone valuations on $[m]$ which admits no allocation giving every agent her $c$-thinned $q_c$-quantile share. Then, proceeding exactly as in that proof, we obtain $\gamma \in (0,c)$ such that for every active agent $i \in I$,
\[\tau_{q_{\gamma}}^{(\gamma)}(u_i,n) = 1.\]
As before, set
\[
k:=\left\lfloor \frac{\gamma}{n}\bigl(M+M^{2/3}\bigr)\right\rfloor.
\]
The crucial point is that since $\gamma<1$, for all sufficiently large $M$ we have
\[
nk\le \gamma(M+M^{2/3})<M.
\]
Thus \cref{conj:rainbow} applies to $k$-uniform families on $[M]$. 

From this point onward, the proof is identical. There is only one thing to check, namely, that the first term in \cref{conj:rainbow} dominates the second term, i.e., that
\[
\binom{kn-1}{k}
\le
\binom{M}{k}-\binom{M-n+1}{k},
\]
for all sufficiently large $M$.
To see this, note that
\[
\binom{M}{k}-\binom{M-n+1}{k}
\ge
\binom{M-1}{k-1}
=
\frac{k}{M}\binom{M}{k} = \Omega_{\gamma,n}\left(\binom{M}{k}\right),
\]
while
\[
{\binom{kn-1}{k}}
\le
\left(\frac{kn}{M}\right)^k {\binom{M}{k}} = o_{M}\left(\binom{M}{k}\right),
\]
where the final bound uses $kn/M\to\gamma<1$ as $M \to \infty$ and $k \to \infty$ as $M \to \infty$. 

Optimality follows from \cref{prop:upper-bound} with $c = 1$.
\end{proof}

\bibliography{main}
\bibliographystyle{amsplain0}

\end{document}